\documentclass[12pt]{amsart}

\usepackage{xcolor}
\usepackage{tikz}
\theoremstyle{definition}
\newtheorem{definition}{Definition}
\theoremstyle{plain}
\newtheorem{theorem}{Theorem}
\newtheorem{corollary}{Corollary}
\newtheorem{lem}{Lemma}
\newtheorem{prop}{Proposition}
\newtheorem{ejem}{Example}
\theoremstyle{remark}

\newtheorem{remark}{Remark}
\DeclareMathOperator{\CC}{\mathbb C}
\DeclareMathOperator{\RR}{\mathbb R}

\DeclareMathOperator{\NN}{\mathbb N}

\DeclareMathOperator{\supp}{Supp}

\title{Zeros of Sobolev orthogonal polynomials and some matrix inequalities. 
 }
\author[1,2]{ Escribano, C.$^{1,2}$}

\author[1]{ Gonzalo, R.$^{1}$}

\address{ Departamento de Matem\'atica Aplicada, Facultad de Inform\'atica de Madrid\\
        Universidad Polit\'ec\-ni\-ca, Campus de Montegancedo\\
      Boadilla del Monte, 28660 Madrid Spain, Phone: +34913367429 \\
      }

\email{cescribano@fi.upm.es}
\email{rngonzalo@fi.upm.es}

\begin{document}

\begin{abstract} The main aim of this work is to apply the  matrix approach of ortho\-gonal polynomials associated with infinite Hermitian definite positive matrices in relation with an important question regarding the  location of zeros of Sobolev orthogonal polynomials via the study of the boundedness of multiplication operator. 
We apply the notion of bounded point evaluations of a measure, and more generally to infinite HPD matrices, to the problem of boundedness of multiplication operator. Moreover, we introduce certain Wirtinger-type inequalities, relating the norm of the polynomials with the norm of their derivatives, in order to provide new examples of Sobolev polynomials for which we may ensure that the zeros of Sobolev polynomials are uniformly bounded. In particular, we consider the case of Lebesgue measures supported on circles. With these techniques, we may obtain many examples of vectorial measures such that the zeros of Sobolev orthogonal polynomials are bounded, and nevertheless they are not sequentially dominated, not even matrix sequentially dominated. 
\end{abstract}

\maketitle
\begin{quotation} {\sc {\footnotesize Keywords}}. {\small Hermitian moment problem,  Sobolev inner products, orthogonal polynomials,
 measures, matrix inequalities}
\end{quotation}

\section{Introduction}

\noindent The problem of the study of the location of zeros of orthogonal polynomials  with respect to Sobolev inner products of the form: 
\begin{equation}\label{sobolev-inner}
<p(z),q(z)=\sum_{k=0}^{N} \int p^{(k)}(z)\overline{q^{(k)}(z)}d\mu_k,
\end{equation}
\noindent where each $\mu_k$ is a finite Borel measure supported on the complex plane and $\mu_0$ is infinitely supported, has attracted much attention in the last years (see e.g. \cite{LP},\cite{LPP}). We are concerned here with the problem of boundedness of the set of zeros of Sobolev orthogonal polynomials in the case that all the measures involved are compactly supported. 

\bigskip 

\noindent Concerning this problem, for the standard case, that is, for orthogonal polynomials associated with a measure $\mu$, the problem of boundedness of the set of zeros is closely related to the boundedness of the multiplication operator on the space of polynomials $\mathbb{P}[z]$,  $\mathcal{D}: \mathbb{P}[z]\to \mathbb{P}[z]$ given by $\mathcal{D}(p)=zp(z)$,   with respect to the $L^{2}(\mu)$ norm. Indeed, for measures supported in the real line, the boundedness of zeros of orthogonal polynomials is equivalent to the boundedness of the multiplication operator, which is indeed equivalent to the boundedness of the support of the measure. This last equivalence, among the boundedness of the support of a measure and the associated multiplication operator, is also true for measures supported on the complex plane. However, boundedness of $\mathcal{D}$  is a sufficient condition for boundedness of the set of zeros, but not a necessary condition. In \cite{Castro-Duran} an easy example is provided of a non-compactly supported measure whose set of zeros of associated orthogonal polynomials is bounded.

\noindent For compactly supported measures on the complex plane,  G. L\'{o}pez and H. Pijeira \cite{LP}, proved that the boundedness of $\mathcal{D}$ with respect to the Sobolev norm is a sufficient condition for the boundedness of the set of zeros of Sobolev orthogonal polynomials. Note that if $\mathcal{D}$ is bounded, then all the measures involved $\mu_k's$ are compactly supported (see, e.g. \cite{Castro-Duran} ). Therefore, one way to approach the problem of boundedness of zeros of Sobolev polynomials is by studying the  boundedness of $\mathcal{D}$. Moreover, in \cite{LP}  the notion of sequentially dominated measures is introduced, that is, when measures $\mu_k=w_kd\mu_{k-1}$ with $w_k\in L^{\infty}(\mu_{k-1})$,  as a sufficient condition to ensure the boundedness of $\mathcal{D}$. 
For non-compactly supported measures in \cite{Duran-Saff} some results of location of zeros are given also using the notion of sequentially dominated measures. Later, J. Rodríguez ( see e.g. \cite{JR1},\cite{JR3}) gave a characterization of boundedness of multiplication operator in terms of certain comparable Sobolev norms. In fact, he proved that $\mathcal{D}$ is bounded with respect to the Sobolev norm if and only if this norm is equivalent to a certain {\it sequential} inner Sobolev norm. 

\bigskip  

\noindent  We here consider the more general context of orthogonal polynomials associated with an Hermitian positive definite matrix, and matrix Sobolev inner products, which are the generalization of inner Sobolev products with respect to measures, using the matrix approach and techniques introduced in \cite{EGT6}, \cite{EG23}. 

\bigskip
\noindent  Recall that any infinite Hermitian positive definite matrix $\mathbf{M}$ (in short HPD matrix) induces an inner product on $\mathbb{P}[z]$ in the following way:
for $p(z)=\sum_{k=0}^{n}v_kz^k, 
q(z)=\sum_{k=0}^{m}w_kz^k \in \mathbb{P}[z]$, 

\begin{equation}\label{aste}
    \langle p(z)
,q(z) \rangle_{\mathbf{M}}=v\mathbf{M}w^{*},
\end{equation}

\noindent where $v=(v_0,v_1,\dots,v_n,0,\dots,0,\dots), w=(w_0,w_1,\dots,w_m,0,\dots)\in c_{00}$, being   $c_{00}$  the space
of all complex sequences with only finitely many non-zero entries. 
In this way, the connection of measures with the matrix approach is via the associated moment matrices. We always consider Borel finite compactly supported on the complex plane measures $\mu$ and the moment matrix is $\mathbf{M}(\mu)=(c_{i,j})_{i,j=0}^{\infty}$ defined by 
$$
c_{ij}= \int z^{i}\overline{z}^{j} d\mu .
$$ 
\bigskip

\noindent  In this general context, concerning the problem of the boundedness of zeros of orthogonal polynomials, in \cite{EGT6} it is shown that the boundedness of $\mathcal{D}$
with respect to  (2)  is also a sufficient condition for the boundedness of the zeros. Moreover, 
in \cite{EG23} the problem of boundedness of zeros of Sobolev orthogonal polynomials is studied in a more general framework than in (1) which are the matrix Sobolev inner products. In the same way as in \cite{LP} a notion of {\it sequentially dominance} of matrices is introduced, and it is proved that this notion is indeed a sufficient condition for boundedness of $\mathcal{D}$, and consequently of the set of zeros of associated orthogonal polynomials. 

\noindent 

\smallskip 
\noindent In this work we focus in the case of Sobolev norms associated to a vectorial measure $(\mu_0,\mu_1)$ and more generally, associated to a pair of infinite HPD matrices $\mathbf{M}_0,\mathbf{M}_1$  of the type 
\begin{equation}{\label{dosmatrices}}
\Vert p(z)\Vert^{2}_{\mathbf{M}_{\mathbf{S}}}= \Vert p(z)\Vert^{2}_{\mathbf{M}_0}+ \Vert p'(z)\Vert^{2}_{\mathbf{M}_1}.
\end{equation}

\noindent We study the following problem, which generalizes the one of boundedness of $\mathcal{D}$ with respect to the vectorial measure when all the measures involved are compactly supported.

\medskip

\noindent {\bf Problem:}
Assume that $\mathcal{D}$ is bounded on $(\mathbb{P}[z], \Vert \cdot \Vert_{\mathbf{M_k}})$ for $k=0,1$. Is $\mathcal{D}$ bounded on $(\mathbb{P}[z], \Vert \cdot \Vert_{\mathbf{M_S}})$ ?

\medskip

\noindent In most of the results previously obtained on this topic, the only technique used to guarantee boundedness is {\it sequentiality}, either of measures or associated moment matrices. Our motivation in this work is to obtain conditions that are essentially different from those of sequentiality.

\noindent As a consequence of a characterization of the boundedness of $\mathcal{D}$  that we will prove here, we should look for conditions to assure the existence of a constant $C>0$ such that 

\begin{equation}\label{desigualdad}
    \Vert p(z) \Vert^{2}_{\mathbf{M}_1} \leq C(\Vert p(z) \Vert^{2}_{\mathbf{M}_0}+ \Vert p'(z) \Vert^{2}_{\mathbf{M}_1}), \qquad p(z) \in \mathbb{P}[z]. 
\end{equation}

\smallskip 

\noindent One could try to obtain inequalities of the type 
$$
\Vert p(z) \Vert^{2}_{\mathbf{M}_1} \leq C\Vert p(z) \Vert^{2}_{\mathbf{M}_0}, \qquad p(z) \in \mathbb{P}[z]. 
$$

\noindent This is the idea behind the notions of sequentiality of measures, or even more so in matrices to ensure boundedness of the multiplication operator. 

\smallskip
 
\noindent However, one could try to find inequalities involving the norms of the polynomials and their derivatives, that is, conditions of the type: 
\begin{equation}\label{Wirtinger}
  \Vert p(z) \Vert^{2}_{\mathbf{M}_1} \leq C\Vert p'(z) \Vert^{2}_{\mathbf{M}_1}, 
\end{equation}
\noindent for polynomials verifying $p(0)=0$. In this sense,   the classical Wirtinger inequality states that there exists a constant $C>0$ such that for every $f(t)\in \mathcal{C}^{1} [a,b]$ with
$f(a)=f(b)$ it follows 
$$
\int_{a}^{b} \vert f(t) \vert^{2} dt \leq C \int_{a}^{b}  \vert f'(t) \vert^{2} dt.
$$

\noindent Here we approach the study, from a matrix point of view,  of this type of inequalities in a certain class of polynomials that vanish at a certain point, which we will call {\it polynomial Wirtinger-type  inequalities}. This, combined with the study of bounded point evaluations of a measure, or more generally of an HPD matrix as in \cite{EGT5}, will allow us to obtain new examples of Sobolev inner products for which the zeros of the associated orthogonal polynomials are bounded. Furthermore, for all these examples obtained, there is no sequentiality of the involved measures, or even matrix sequentiality of the same.

\bigskip
\noindent The structure of the paper is as follows. 

\medskip

\noindent In Section 1 we introduce matrix Sobolev inner products as in \cite{EG23} for a set of two HPSD matrices with the first positive definite. The main result will be a simple and useful characterization of the boundedness of the multiplication operator that generalizes the one obtained by J. Rodríguez (see, eg. \cite{JR1})  for Sobolev inner productos associated with measures.

\medskip

\noindent In the second section, we will study the bounded point evaluations of a matrix and the impact of their existence on the problem of boundedness of $\mathcal{D}$. 
We provide several characterizations of the bounded point evaluations associated with an HPD matrix
in terms of a certain matrix index introduced in \cite{EGT6}.
As a consequence of these results we obtain examples of vectorial measures which are matrix sequentially dominated and consequently the set of zeros of Sobolev polynomials is bounded. These results are more precise when we consider vectorial measures where the first measure is supported in a Jordan curve, for which polynomials are not dense in the corresponding $L^{2}$ spaces, and the second measure is supported in the interior of the curve.  

  \medskip
\noindent The third section is devoted to the study of measures that verify certain polynomial Wirtinger-type inequalities using a matrix approach. Such inequalities hold for instance for the Lebesgue measure in the unit circle ${\bf m}$, and more generally for measures of the type $w(\theta)d{\bf m}$ with $w(\theta)\in L^{\infty}({\bf m})$ and ${\it ess\; inf w(\theta)>0}$. Moreover, we prove that if the associated moment matrix of a measure verifies a certain polynomial Wirtinger-type inequality (4) with constant $C=1$ it must be of the type $c{\bf m}$. 

\medskip 

\noindent Finally, combining the results of the previous sections, we obtain new examples of vectorial measures for which the zeros of the associated orthogonal polynomials are bounded and yet they are not sequentially dominated,  not even matrix sequentially dominated. These examples are closely related to inner products, where the second measure is the Lebesgue measure supported on circles or unions of circles. Moreover, using these techniques, we give an example of two vectorial measures whose associated norms are comparable, in the sense that they are equivalent in the vector space of polynomials, and yet the involved measures are not comparable component by component (that is the associated norms are not equivalent in the respective $L^2$ spaces). 

\bigskip
\noindent Some notations and definitions: 

 \noindent  In an analogous way to the case of HPD matrices, for an infinite positive semidefinite matrix $\mathbf{M}$ (in short HPSD matrix) the expression $(2)$ defines an Hermitian sesquiliar form that induces a seminorm that we denote in the same way.



\begin{definition} Let $\mathbf{M}_0,\mathbf{M}_1$  be two HSPD  matrices. We say that
 $\mathbf{M}_1 \leq \mathbf{M}_0$ if
$
v\mathbf{M}_1v^{*} \leq v \mathbf{M}_0v^{*}$,    for every $v\in c_{00}$.

\end{definition}


\noindent Note that if $\mathbf{M}_1 \leq C\mathbf{M}_0$ for some $C>0$, then 
$$
\Vert p(z) \Vert_{\mathbf{M}_1} \leq C \Vert p(z)\Vert_{\mathbf{M}_0}, \qquad p(z) \in \mathbb{P}[z],
$$
\noindent where the expressions $\Vert \cdot \Vert_{\mathbf{M_i}}$ for $i=0,1$ are seminorms if the associated matrices are HPSD. 

\bigskip
\noindent Thorough  the paper all the measures considered are Borel finite measures supported in the complex plane. We always consider vectorial measures $\mu=(\mu_0,\mu_1)$ with the involved measures compactly supported and with the first measure $\mu_0$  infinitely supported. For $\mu=(\mu_0,\mu_1)$, we consider the Sobolev inner product on $\mathbb{P}[z]$,

\begin{equation}{\label{dosmedidas}}
<p(z),q(z)>_{\mathbf{S}}= \int \vert p(z) \vert^{2} d\mu_0 + \int \vert p'(z) \vert^{2} d\mu_1.
\end{equation}

\noindent We denote by $\Vert \cdot \Vert_{\mathbf{S}}$ the associated Sobolev norm with $(\mu_0,\mu_1)$ on $\mathbb{P}[z]$. 

\section{Characterization of boundedness of multiplication operator for matrix Sobolev inner products}
\noindent We recall the definition of matrix Sobolev inner products introduced in \cite{EG23}:  

\begin{definition} Let $\{\mathbf{M}_0,\mathbf{M}_1\}$ be a set of infinite HPSD matrices  with $\mathbf{M}_0$ being HPD,  the induced matrix Sobolev inner  product, denoted by $<\cdot,\cdot>_{\mathbf{M_S}}$,  is defined in the following way:
\begin{equation}{\label{dosmatrices}}
<p(z),q(z)>_{\mathbf{M}_{\mathbf{S}}}= <p(z),q(z)>_{\mathbf{M}_0}+ <p'(z),q'(z)>_{\mathbf{M}_1}, \quad p(z),q(z)\in \mathbb{P}[z].
\end{equation}

\noindent We denote  by $\Vert \cdot \Vert_{\mathbf{M_S}}$ the induced matrix Sobolev norm, that is, 

$$
\Vert p(z) \Vert^{2}_{\mathbf{M}_{\mathbf{S}}}= 
\Vert p(z) \Vert^{2}_{\mathbf{M}_0} + \Vert  p'(z)\Vert^{2}_{\mathbf{M}_1}.
$$
\end{definition}



\noindent Following \cite{EG23}, a set of HPSD matrices $\{\mathbf{M}_0,\mathbf{M}_1\}$ is a set of {\it sequentially dominated matrices} if there exists $C>0$ such that $\mathbf{M}_1 \leq C \mathbf{M}_0$. In (\cite{EG23} Theorem 1), it is shown that if a set of HPD matrices $\{\mathbf{M}_0,\mathbf{M}_1\}$ is sequentially dominated, then $\mathcal{D}$ is bounded with respect to the matrix Sobolev inner product (\ref{dosmatrices}). The same proof works with $\mathbf{M_1}$ being HPSD.

\bigskip

\noindent We provide a generalization of the characterization of the boundedness of the multiplication operator given in \cite{JR1} in the context of matrix Sobolev inner products (\ref{dosmatrices}). Although we use the same arguments, we include the proof for the sake of completeness.

\begin{prop} Let $\Vert \cdot \Vert_{\mathbf{M_S}}$ be a matrix Sobolev norm associated with a pair of HPSD matrices  $\{\mathbf{M}_0,\mathbf{M}_1\}$  with $\mathbf{M}_0$ being HPD. Assume that $\mathcal{D}$ is bounded on 
$(\mathbb{P}[z],\Vert \cdot \Vert_{\mathbf{M}_0})$ and  there exists $C_1>0$ such that 
$$
\Vert zp(z) \Vert^{2}_{\mathbf{M}_1} \leq C_1 
\Vert p(z) \Vert^2_{\mathbf{M}_1
},\qquad p(z) \in \mathbb{P}[z].
$$
\noindent   Then, the following are equivalent: 
\begin{enumerate} 
\item  $\mathcal{D}$ is bounded on $(\mathbb{P}[z],\Vert \cdot \Vert_{\mathbf{M_S}})$.  

\item There exists  $C>0$ such that

\begin{equation}{\label{condicion}}
\Vert p(z)\Vert_{\mathbf{M}_{1} } \leq C\Vert p(z) \Vert_{\mathbf{M}_{\mathbf{S}}}, \qquad p(z) \in \mathbb{P}[z].
\end{equation}
\end{enumerate}

\end{prop} 



\begin{proof} For every polynomial $p(z)\in \mathbb{P}[z]$

$$
\Vert \mathcal{D}(p(z)) \Vert^2_{\mathbf{M}_{\mathbf{S}}}=  \Vert zp(z) \Vert^2_{\mathbf{M}_0} + \Vert (zp(z))' \Vert^2_{\mathbf{M}_1}.
$$
\bigskip

\noindent First, since $\mathcal{D}$ is bounded on $(\mathbb{P}[z], \Vert \cdot \Vert_{\mathbf{M}_0})$ by taking  $R_0= \Vert \mathcal{D}_0 \Vert $ we have that 

$$\Vert zp(z) \Vert^2_{\mathbf{M}_0} \leq R^2_{0} \Vert p(z)\Vert^2_{\mathbf{M}_0} \leq R_0^2 \Vert p(z)\Vert^2_{\mathbf{M}_{\mathbf{S}}}.
$$

\noindent On the other hand 
$$
\Vert (zp(z))' \Vert^2_{\mathbf{M}_1}= \Vert p(z) + zp'(z) \Vert^{2}_{\mathbf{M}_1} \leq (\Vert p(z) \Vert_{\mathbf{M}_1} + \Vert zp'(z) \Vert^{2}_{\mathbf{M}_1} )^2 
$$

$$ \leq 
2( \Vert p(z) \Vert_{\mathbf{M}_1}^{2} + \Vert zp'(z) \Vert^{2}_{\mathbf{M}_1} )\leq
2\Vert p(z) \Vert_{\mathbf{M}_1}^{2} + 2C_1^2\Vert p'(z) \Vert^{2}_{\mathbf{M}_1} \leq 2(C^2+C_1^2) \Vert p(z) \Vert_{\mathbf{M}_{\mathbf{S}}}^{2} 
$$

\smallskip

\noindent By considering  $R=R_{0}^{2}+2(C^2+C_1^2) $ have that for all $p(z)\in \mathbb{P}[z]$ 
$$
\Vert \mathcal{D}(p(z)) \Vert^2_{\mathbf{M}_{\mathbf{S}}}= \Vert zp(z)\Vert_{\mathbf{M}_{\mathbf{S}}}^{2} \leq R \Vert p(z) \Vert^2_{\mathbf{M}_{\mathbf{S}}},
$$

\noindent consequently 
$\mathcal{D}$ is bounded on $(\mathbb{P}[z], \Vert \cdot \Vert_{\mathbf{M_S}})$ as we required. 
\end{proof}

\noindent We now obtain the generalization in the more general context of matrix Sobolev norms of the characterization of the boundedness of the multiplication operator given by \cite{JR3}.

\begin{theorem} Let $\Vert \cdot \Vert_{\mathbf{M_S}}$ be a matrix Sobolev norm associated with a pair of HPSD matrices  $\{\mathbf{M}_0,\mathbf{M}_1\}$  with $\mathbf{M}_0$ being HPD. Assume that $\mathcal{D}$ is bounded on 
$(\mathbb{P}[z],\Vert \cdot \Vert_{\mathbf{M}_0})$. Then, the following are equivalent: 
\begin{enumerate}
\item $\mathcal{D}$ is bounded on $(\mathbb{P}[z], \Vert \cdot \Vert_{\mathbf{M_S}})$.
\item $\Vert \cdot \Vert_{\mathbf{M_S}}$ is equivalent to $\Vert \cdot \Vert_{\mathbf{M_{S'}}}$ on $\mathbb{P}[z]$ which is $\Vert \cdot \Vert_{\mathbf{M_{S'}}}$ the matrix Sobolev inner product associated with $\{\mathbf{M}_0+\mathbf{M}_1,\mathbf{M}_1\}$
\end{enumerate}
\end{theorem}

\begin{proof} Assume $\mathcal{D}$ is bounded on $(\mathbb{P}[z], \Vert \cdot \Vert_{\mathbf{M_S}})$.  First, since $\mathbf{M}_0 \leq \mathbf{M_0}+\mathbf{M}_1$ it holds 
$$
\Vert p(z) \Vert^{2}_{\mathbf{M_S'}}= \Vert p(z) \Vert^{2}_{\mathbf{M}_0+\mathbf{M}_1}+ \Vert p'(z) \Vert^{2}_{\mathbf{M}_1} \geq 
\Vert p(z) \Vert^{2}_{\mathbf{M}_0}+ \Vert p'(z) \Vert^{2}_{\mathbf{M}_1} = \Vert p(z) \Vert^{2}_{\mathbf{M_{S}}}.
$$
\noindent On the other hand, 
$$
\Vert p(z) \Vert^{2}_{\mathbf{M_{S'}}}= \Vert p(z)\Vert^{2}_{{\mathbf{M_0}+\mathbf{M}_1}}+ \Vert p'(z) \Vert^{2}_{\mathbf{M_1}} = \Vert p(z) \Vert^{2}_{\mathbf{M_0}}+\Vert p(z) \Vert^{2}_{\mathbf{M_1}}+ \Vert p'(z) \Vert^{2}_{\mathbf{M_1}}, 
$$

\noindent by Proposition 1 it follows (\ref{condicion})  and, 
$$
\Vert p(z) \Vert^{2}_{\mathbf{M_{S'}}}\leq    C^2\Vert p(z) \Vert^{2}_{\mathbf{M_S}}+ \Vert p(z) \Vert^{2}_{\mathbf{M}_0} +  \Vert p'(z) \Vert^{2}_{\mathbf{M_1}}
 \leq (C^2+1)\Vert p(z) \Vert^{2}_{\mathbf{M_S}}.
$$

\noindent Therefore, 
$$
\Vert p(z) \Vert^{2}_{\mathbf{S}}
\leq \Vert p(z) \Vert^{2}_{\mathbf{S'}}
\leq (C^2+1)\Vert p(z) \Vert^{2}_{\mathbf{M_S}}, \qquad p(z) \in \mathbb{P}[z],
$$

\noindent and the norms  $\Vert p(z) \Vert^{2}_{\mathbf{M_{S}}},  \Vert p(z) \Vert^{2}_{\mathbf{M_{S'}}}$ are equivalent on $\mathbb{P}[z]$. 

\noindent In order to prove $(2)$ implies $(1)$, since $\{\mathbf{M}_0+\mathbf{M}_1, \mathbf{M}_1\}$ is sequentiality dominated by \cite{EG23} it follows that $\mathcal{D}$ is bounded on $(\mathbb{P}[z], \Vert \cdot \Vert_{\mathbf{S'}})$.  
Since an operator which is bounded with respect to a norm is bounded with respect to any equivalent norm, we have the required conclusion.

\end{proof}

\noindent  Note that for vectorial measures $(\mu_0,\mu_1)$, when we particularize Proposition 1 for the set of associated moment matrices $\{\mathbf{M}(\mu_0), \mathbf{M}(\mu_1)\}$, we obtain the following result which is the characterization of boundedness of $\mathcal{D}$ in \cite{JR1}, \cite{JR3}. 

\begin{corollary} Let $\mu=(\mu_0,\mu_1)$  be a vectorial measure and  the Sobolev inner product (\ref{dosmedidas}). Then, the following are equivalent: 

\begin{enumerate} 
\item The multiplication operator  is bounded with respect to $\Vert \cdot \Vert_{\mathbf{S}}$, 
\item There  exists a constant $C>0$ such that,
$$
\int \vert p(z) \vert^{2} d\mu_1\leq C (\int \vert p(z) \vert^{2}d\mu_0 + \int \vert p'(z) \vert^{2} d\mu_1) \qquad p(z) \in \mathbb{P}[z],
$$
\item $\mathcal{D}$
is bounded with respecto to the Sobolev norm induced by the vectorial measure $(\mu_0+\mu_1,\mu_1)$.
\end{enumerate}
\end{corollary}

\section{Bounded point evaluations and zeros of Sobolev orthogonal polynomials} 



\noindent Recall that $a\in \CC$ is a bounded point evaluation of a measure $\mu$, in short a bpe of $\mu$,  if there exists $C>0$ such that $$ \vert p (a) \vert 2 \leq C \int \vert p(z) \vert^2 \; d\mu, \qquad p(z)\in \mathbb{P}[z].
$$

\noindent We denote by $bpe(\mu)= \{a\in \mathbb{C}: a \text{ is a bpe of  }  \mu\}$. It is known  (see e.g.   \cite{Conway}) that 
for a compact set  $K\subset \mathbb{C}$  the {\it polynomially convex 
hull } of $K$, denoted by $P_C{(K)}$, is defined as
$$
P_C(K):=\{ z\in \CC:  \; \vert p(z) \vert \leq \max_{\xi \in K}\vert p(\xi) \vert,  \; \text{for all } p(z)\in \mathbb{P}[z] \}.
$$

\noindent A compact set is said to be {\it polynomially convex} if $K=P_C(K)$. Obviously, $K \subset P_C(K)$.  Using similar techniques as in \cite{EGT6}, we here prove that the bounded point evaluations of a measure are contained in the polynomially convex hall of its support.

\begin{lem} Every bounded point evaluation of a measure $\mu$ is contained in the polynomially convex hall of $\mu$. 
\end{lem} 
\begin{proof} Assume that $a$ is a bpe for $\mu$, that is, there exists $C>0$ such that
$$
\vert p(a) \vert^2 \leq C\int \vert p(z) \vert^2 \; d\mu \qquad p(z)\in \mathbb{P}[z] .
$$

\noindent Then $a\in P_C({\it supp}(\mu))$. Assume to the contrary, then by \cite{Forstneric} there exists $p(z)\in \mathbb{P}[z]$ such that $p(a)=1$ and $\vert p(z)\vert<1$
for all $z\in P_C({\it supp}(\mu))$. Since $P_C({\it supp}(\mu))$ is a compact set there exists $0<\alpha <1$ such that $\vert p(z) \vert \leq \alpha$ for every $z\in P_C({\it supp}(\mu))$. By taking  $q_n(z)=p(z)^n$  it follows 
$$
1 \leq C\int \vert p(z) \vert^{2n} \; d\mu \leq \alpha^{2n}\;\mu({\it supp}(\mu)),
$$
\noindent which is  not possible. 
\end{proof} 

\noindent In \cite{EGT5} the following notion of bpe of an HPD is introduced: 
\begin{definition} Let $\mathbf{M}$ be an HPD matrix and  $a\in \CC$, we say that $a$ is a bounded point evaluation (in short a bpe)   of $\mathbf{M}$ if there exists  $C>0$ such that 

$$
\vert p(a) \vert^{2} \leq C \Vert p(z) \Vert^{2}_{\mathbf{M}}, \qquad p(z) \in \mathbb{P}[z].
$$
\end{definition}

\begin{remark} Note that a bounded point evaluation of a measure $\mu$ is,   indeed, a bounded point evaluation of the associated moment matrix  $\mathbf{M}(\mu)$. 
\end{remark}

\begin{remark} When we consider the probability Dirac measure $\delta_{a}$ concentrated at a point $a\in \CC$ and  its associated moment matrix $\mathbf{M(a)}=\mathbf{M}(\delta_{a})$, the fact of being a bpe of  $\mathbf{M}(a)$ turns to be that 
$\{\mathbf{M}, \mathbf{M(a)}\}$ is a set of sequentiality dominated HPSD matrices. 
\end{remark}

\noindent Using the matrix approach for  bounded point evaluations, this notion is related with a pointwise matrix index $\gamma_{a}$ introduced in \cite{EGT5} as follows: let $\mathbf{M}$ be an infinite HPSD matrix and $a\in \CC$, 
$$
\gamma_{a}(\mathbf{M}):=\inf \{ v\mathbf{M}v^{*}: \sum_{k=0}^{\infty} v_ka^{k}=1, v\in c_{00}\}. 
$$
\noindent In \cite{EGT5} it is proved that $a$ is a bpe of an HPD matrix $\mathbf{M}$ if and only if $\gamma_{a}(\mathbf{M}>0$. 
 In the following proposition, we summarize these results for bpe's of a matrix:

\begin{prop} Let $\mathbf{M}$ be an HPD matrix. Then, the following are equivalent:

\begin{enumerate}
\item $a$ is a bounded point evaluation of $\mathbf{M}$.


\smallskip 

\item The set $\{\mathbf{M},\mathbf{M}(a)\}$ is sequentially dominated. 
\item $\gamma_{a}(\mathbf{M})>0$.
\end{enumerate}
\end{prop}

%
 
%
\bigskip

\noindent In the sequel we consider discrete matrix Sobolev norms   of the type
\begin{equation}\label{sobPoint}
    \Vert p(z)\Vert^{2}_{\mathbf{M}_{\mathbf{S}}}:=\Vert p(z)\Vert^{2}_{\mathbf{M}}+ \sum_{k=1}^{n} \vert p'(a_k) \vert^{2},
\end{equation}

\noindent where $\mathbf{M}$ is an HPD matrix and $a_1,\dots,a_n\in \CC$. This notion is a generalization of discrete Sobolev inner products, this is, associated with $\mu$  infinitely supported  in the complex plane and 
$a_1,\dots ,a_n\in \CC $, 
\begin{equation} \label{discrete}
\Vert p(z) \Vert^2_{\mathbf{S}}= \int \vert p(z) \vert^{2} d\mu + \sum_{k=1}^{n} \vert p'(a_k) \vert^{2}.
\end{equation}
\noindent

\begin{prop} Let $\mathbf{M}$ be an HPD matrix such that  $\mathcal{D}$ is bounded on $(\mathbb{P}[z], \Vert \cdot \Vert_{\mathbf{M}})$ and let $a_1,\dots ,a_n\in \CC $ be bounded point evaluations of $\mathbf{M}$.  Then  the multiplication operator $\mathcal{D}$ is bounded on $\mathbb{P}[z]$ with respect to the   inner norm (\ref{sobPoint}).

\end{prop} 
\begin{proof} Since $a_k$ is a bounded point evaluation of $\mathbf{M}$, for every $k=1,\dots,n$,  then there exist  $C_k>0$ such that 
$$
<p(z),p(z)>_{\mathbf{M(a_k)}} \leq C_k <p(z),p(z)>_{\mathbf{M}}, \quad p(z)\in \mathbb{P}[z], \quad k=1,\dots,n.
$$

\noindent Then, the set $\{\mathbf{M},\sum_{k=1}^{n} \mathbf{M(a_k)}\}$ is sequentiality dominated, indeed, if $C=\sum_{k=1}^{n} C_k$
$$
<p(z),p(z)>_{\sum_{k=1}^n \mathbf{M}(a_k)} = \sum_{k=1}^{n}<p(z),p(z)>_{\mathbf{M(a_k)}} \leq C <p(z),p(z)>_{\mathbf{M}}.
$$

\noindent Since $\Vert \cdot \Vert_{\mathbf{M_S}}$ is the inner product associated with $\{\mathbf{M},\sum_{k=1}^{n} \mathbf{M(a_k)}\}$ the result is a consequence of Proposition 1. 
\end{proof}
\noindent 





\noindent In the sequel, we particularize the  above result for the case of  discrete Sobolev measures as in (\ref{discrete}).

\begin{corollary} Let $\mu$ be a measure with infinite support and let 
$a_1,\dots ,a_n\in bpe(\mu)$.  Then $\mathcal{D}$ is bounded with respect to the discrete Sobolev norm (\ref{discrete}). Consequently, the set of zeros of associated Sobolev orthogonal polynomials is bounded.
\end{corollary}
\noindent  More interesting and precise results are obtained in the particular case of measures supported on Jordan curves. Indeed, by using the results in \cite{EGT5} concerning the localization of bpe's of a measure in terms of denseness of polynomials, we have:

\begin{corollary} Let $\nu_{\Gamma}$ be a measure supported on a  closed  Jordan curve $\Gamma$ verifying  that polynomials are not dense in 
$L^{2}(\nu_{\Gamma})$. Assume that  $a_1,\dots,a_{n} \in int(\Gamma)$, and consider  the associated  discrete Sobolev norm
$$
\Vert p(z) \Vert^2_{\mathbf{S}}= \int \vert p(z) \vert^{2}d\nu_{\Gamma} + \sum_{k=1}^{n} \vert p'(a_k)\vert^{2}.
$$

\noindent Then, the  $\mathcal{D}$  is bounded on $(\mathbf{P}[z],\Vert \cdot \Vert_{\mathbf{S}})$. Consequently, the set of zeros of associated Sobolev orthogonal polynomials is bounded.

\end{corollary} 
\begin{proof} By using  Corollary 6 in  \cite{EGT5}, it follows that every point in the interior of $\Gamma$ is a bpe of $\nu_{\Gamma}$, consequently by Corollary 2 the result holds. 

\end{proof}

\noindent In the sequel, we improve the above result by removing the atomic measure by a measure whose support is contained in the interior of $\Gamma$. In order to do it, we need the following technical lemma:

\begin{lem} Let $\mu$ be a measure with infinite  support and let $\mathbf{M}(\mu)$ be the associated moment matrix. Let $K$ be a compact set with  $K \subset \mathbb{R}^{2}\setminus {\it supp}(\mu)$ and $\gamma_{z_0}(\mathbf{M})>0$ for every $z_0\in K$. Then, there exists a constant $c_{K}>0$
such that $\gamma_{z_0}(\mathbf{M})\geq c_{K}$ for every $z\in K$. 
\end{lem}

\begin{proof}   Let  $\{P_{k}(z)\}_{k=0}^{\infty}$ be  the sequence of orthonormal polynomials associated with $\mu$.  By using the results in \cite{EGT1} it holds that if 
$w \notin {\it supp}(\mu)$, in particular when $w\in K$ it follows,  
$$
\dfrac{1}{D_{w}^2\sum_{k=0}^{\infty} \vert P_k(w)\vert^2} \leq {\it dis}^2 \left( \dfrac{1}{z-w}, \mathbb{P}[z] \right)
 \leq \dfrac{1}{d_{w}^2\sum_{k=0}^{\infty} \vert P_k(w)\vert^2},
$$
\noindent where $d_{w}=\min \{ \vert z - w \vert: z\in {\it supp}(\mu)\}$, $D_{w}= \max \{ \vert z - w \vert: z\in {\it supp}(\mu)\}.$

\noindent Consequently, for every $w \in K$ it holds 
$$
d_{w}^{2} {\it dis}^2 \left( \dfrac{1}{z-w}, \mathbb{P}[z] \right) \leq \gamma_{w}(\mathbf{M}) \leq D_{w}^{2}{\it dis}^2 \left( \dfrac{1}{z-w}, \mathbb{P}[z] \right). 
$$

\noindent Consider $r={\it dis}(K,{\it supp}(\mu))= \min \{ \vert z-w\vert : z\in K, w\in {\it supp}(\mu)\}>0$ and  $R={\it max} \{ \vert z-w \vert , z\in {\it supp}(\mu), w\in K \}$,  obviously $r\leq d_{w}$  and $D_{w} \leq R$ for every $w\in K$ and 
$$
r^{2} {\it dis}^2 \left( \dfrac{1}{z-w}, \mathbb{P}[z] \right) \leq \gamma_{w}(\mathbf{M}) \leq R^{2}{\it dis}^2 \left( \dfrac{1}{z-w}, \mathbb{P}[z] \right).
$$

\noindent The  function ${\it dis}(\dfrac{1}{z-w},\mathbb{P}[z])$ is uniformly continuous on $K$ with constant  $\frac{1}{r^2}$ since  

$$
\Vert  \dfrac{1}{z-w_1} -p(z) \Vert \leq \Vert \dfrac{1}{z-w_1}-
\dfrac{1}{z-w_2}\Vert + \Vert \dfrac{1}{z-w_2}-p(z) \Vert .
$$

\noindent And by taking infimum over all polynomials  
$$
{\it dis}(\dfrac{1}{z-w_1}, \mathbb{P}[z])) \leq  \Vert \dfrac{1}{z-w_1}-
\dfrac{1}{z-w_2}\Vert + {\it dis}( \dfrac{1}{z-w_2},\mathbb{P}[z]).
$$
\noindent Therefore 
$$
\vert {\it dis}(\dfrac{1}{z-w_1}, \mathbb{P}[z]) - {\it dis}(\dfrac{1}{z-w_2}, \mathbb{P}[z]) \vert \leq \Vert \dfrac{1}{z-w_1}-\dfrac{1}{z-w_2}\Vert \leq \dfrac{\vert w_1-w_2\vert}{r^2}.
$$

\noindent By continuity since  $K$ is a compact set we may consider 
 
$$
\alpha  = \min \{ {\it dis}(\dfrac{1}{z-w}, \mathbb{P}[z]), z\in K\} >0.
$$

\noindent Therefore by taking $C_{K}=\alpha r^{2}$ it  follows that for every $w\in K$ 
$$
\gamma_{w}(\mathbf{M}) \geq r^{2}\alpha =C.
$$
\end{proof}

\begin{remark} Note that in this proof is essential the use of a measure in order to estimate the distance of the functions $\dfrac{1}{z-z_0}$ to the vector space of polynomials. We do not know if it is possible to do a proof  without using the measure, only with techniques of  matrix analysis. 
\end{remark}

\begin{corollary} Let $\mu$ be a measure with infinite support and let $K$ be a compact set with $K \subset \mathbb{R}^{2}\setminus {\it supp}(\mu)$. Assume that $K \subset bpe(\mu)$,  then there exists a positive constant $C_{K}$ such that for every $z\in K$
$$
\vert p(z) \vert^{2} \leq C_K \int \vert p(z) \vert^{2} d\mu.
$$
\end{corollary}

\noindent As a consequence of the above results, we provide examples of vectorial measures which are matrix sequentially dominated.

\begin{prop} Let $\mu=(\mu_0,\mu_1)$ be a vectorial measure  with ${\it supp}(\mu_1) \subset bpe(\mu_0)$.
Then $(\mu_0,\mu_1)$ is dominated by matrix sequentiality. 
\end{prop}

\begin{proof} By  Corollary 4, all points in $K$ are bpe's with uniform constant $C_K>0$. Therefore, 

$$ 
\int  \vert p(z) \vert^{2} d\mu_1 \leq \int_{K} \left(C_{K} \int \vert p(z) \vert^{2} d\mu_0 \right)d\mu_1 \leq 
\mu_1(K) C_{K} \int \vert p(z) \vert^{2} d\mu_0 . 
$$

\noindent Consequently, $\mathbf{M}(\mu_1) \leq \mu_1(K) C_{K} \mathbf{M}(\mu_0)$, i.e.,  
 $(\mu_0,\mu_1)$ is matrix sequentiality  dominated.
\end{proof}

\noindent In what follows, we consider some examples of application of the previous results for measures supported on Jordan curves. Bounded point evaluations for this kind of measures have been studied, using the matrix approach, in \cite{EGT5}. There, it is proved that for a measure $\mu$ supported on a Jordan curve verifying that polynomials are not dense in $L^{2}(\mu)$ it follows that all the points in the interior of the curve are bpe's of the measure. Combining these results with Proposition 4 we obtain some examples of vectorial measures which are matrix sequentially dominated: 

\begin{prop} Let $\nu_{\Gamma}$ be a measure supported on a closed Jordan curve $\Gamma$ verifying $P^{2}(\nu_{\Gamma}) \neq L^{2}(\nu_{\Gamma})$. Let $\mu$ be any measure supported in the interior of $\Gamma$. Then the vectorial measure  $\mu=(\nu_{\Gamma},\mu)$  is matrix sequentially dominated and consequently the multiplication operator is bounded with respect to the norm 
$$
\Vert p(z) \Vert^{2} = \int \vert p(z) \vert^{2} d\nu_{\Gamma} + \int \vert p(z) \vert^{2} d\mu .
$$
\end{prop}




\noindent \begin{remark} It is interesting to point out  that  Proposition 5  provides a great variety  of examples of matrix sequentially dominated vectorial measures $\mu=(\mu_0,\mu_1)$ which are not sequentially dominated. For example, vectorial measures where  the first measure is any Lebesgue measure on a circle and the second one is any measure (discrete, singular, etc) whose  support is contained in its interior. 
\end{remark}

\noindent In what follows, we focus on the study of Lebesgue Sobolev orthogonality on the unit circle, or more generally, on arbitrary circles in the complex plane.  Let ${\bf m_{a;r}}$ denote the normalized Lebesgue measure on the circle $\mathbb{S}_{r}(a)=\{z\in \CC: \vert z-a \vert =r\}$ (in the particular case of the unit circle ${\bf m}_{0;1}={\bf m})$. Let $\nu$ be a measure on the unit circle and consider  the  Sobolev norm
$$
\Vert p(z) \Vert^{2}_{\mathbf{S}}= \int \vert p(z) \vert^{2} d \nu + 
 \sum_{k=1}^{n} \int \vert p'(z) \vert^{2} d{\bf m_{a_k;r_k}},
$$
\noindent with $a_1,\dots,a_k\in \CC$ and $r_1,\dots,r_k>0$. In \cite{Cachafeiro} the asymptotics of Sobolev orthogonal polynomials for this type are considered when only the Lebesgue measures in the unit circle are involved. Recall that by the Lebesgue decomposition $\nu=hd{\bf m}+\nu_s$ with $h\in L^{1}({\bf m})$ and $\nu_s$ a singular measure with respect to ${\bf m}$. By the Szeg"{o} theory, it is well known that $P^{2}(\nu)\neq L^{2}(\nu)$ if and only $\nu$ verifies the Szeg\"{o} condition, i.e. $log(h)\in L^{1}({\bf m})$. 

\begin{prop} Let $\nu$ be a measure supported in  the unit circle verifying Szeg\"{o} condition. In the following two cases:
\begin{enumerate} 
\item Discrete Sobolev inner products with $a_1,\dots,a_n\in \mathbb{D}$ 
$$
\Vert p(z) \Vert_{\mathbf{S}} = \int \vert p(z) \vert^{2} d\nu + \sum_{k=1}^{n} \vert p'(a_k) \vert^{2},
$$
\item Continuous Sobolev inner product with $a_1,\dots,a_n\in \CC$ with $\mathbb{D}(a_k;r_k) \subset \mathbb{D}$
$$
\Vert p(z) \Vert^{2}_{\mathbf{S}}= \int \vert p(z) \vert^{2} d \nu + 
 \sum_{k=1}^{n} \int \vert p'(z) \vert^{2} d{\bf m_{a_k;r_k}}
$$
\end{enumerate}
\noindent $\mathcal{D}$ is bounded on $(\mathbb{P}[z], \Vert \cdot \Vert_{\mathbf{S}})$. Consequently, the set of zeros of the Sobolev orthogonal polynomials is bounded. 
\end{prop}

\noindent In the particular case of $\nu={\bf m_{b;r}}$ the situation is as follows. 
\begin{ejem} Let $\mu_0={\bf m}_{b;r}$ 
 and $a_1,\dots,a_N\in \CC$.  Let  $r_1,\dots,r_N>0$ be with $\mathbb{D}(a_k;r_k) \subset \mathbb{D}(b;r)$ and consider the Sobolev norm:  
 $$
 \Vert p(z) \Vert^{2}_{\mathbf{S}} = \int \vert p(z) \vert^{2} d{\bf m}_{b;r}+ \sum_{k=1}^{N} \int \vert p'(z) \vert^{2} d{\bf m}_{a_k;r_k},
 $$
 \noindent $\mathcal{D}$ is bounded with respect this norm and consequently, the 
  set of zeros of Sobolev orthogonal polynomials is bounded. 
 
\end{ejem}






\section{Polynomial Wirtinger-type inequalities.} 

\noindent In this section, we introduce polynomial Wirtinger-type inequalities in the sense of inequalities that relate the norm of a polynomial with the norm of its derivative. We say that a measure $\mu$ verifies a polynomial Wirtinger inequality if there exists $C>0$ 

$$
\int \vert p(z)\vert^2 d\mu \leq  C \int  \vert p'(z) \vert^2 d\mu, \qquad p(z) \in \mathbb{P}_0[z]
$$
\noindent where $\mathbb{P}_{0}[z]= \{p(z)\in \mathbb{P}[z] : p(0)=0\}$. 

\medskip

\noindent More generally,  we say that an HPD matrix $\mathbf{M}$ verifies a Wirtinger polynomial inequality if there exists $C>0$ such that
$$
\Vert p(z) \Vert^{2}_{\mathbf{M}} \leq C \Vert p'(z) \Vert^{2}_{\mathbf{M}}, \qquad p(z) \in \mathbb{P}_0[z].
$$
 We begin with an easy inequality for Lebesgue measures on circles: 

\begin{lem} Let  ${\bf m}_{a;r}$ be the normalized Lebesgue measure on $\mathbb{S}_r(a)$. 
Then, for every $p(z)\in \mathbb{P}[z]$ 

$$
\int \vert p(z) -p(a)\vert^2 d {\bf m}_{a;r} \leq  r^{2} \int  \vert p'(z) \vert^2 d {\bf m}_{a;r}.
$$
\end{lem}
\begin{proof}  Let $p(z) \in \mathbb{P}[z]$, then 
$$
p(z)-p(a)=a_1(z-a)+a_2(z-a)^2+\dots+a_{n}(z-a)^{n}
$$
\noindent Since $\{(z-a)^k\}_{k=0}^{\infty}$ is a sequence of orthogonal polynomials with respect the measure ${\bf m}_{a;r}$ and  
$$
\int (z-a)^{k}\overline{(z-a)^{n}} d{\bf m}_{a;r}= \begin{cases} r^{2k} \qquad  if \; k=n \\ 0 \qquad \; if \; k\neq n,\end{cases} 
$$

\noindent it follows: 
$$
\int \vert p(z) -p(a) \vert^2 d{\bf m}_{a;r} = \sum_{k=1}^{n} \vert a_k \vert^2 r^{2k} =r^2(\sum_{k=1}^{n} \vert a_k \vert^2 r^{2(k-1)}).
$$

\noindent On the other hand, since $p'(z)=a_1+2a_2(z-a)+\dots+na_n(z-a)^{n-1}$ it follows 

$$
\int \vert p'(z) \vert^2 d{\bf m}_{a;r} = \sum_{k=1}^{n} k^2\vert a_k \vert^2 r^{2(k-1)}.
$$

\noindent Therefore, 

$$\int \vert p(z) -p(a)\vert^2 d{\bf m}_{a;r} =r^2(\sum_{k=1}^{n} \vert a_k \vert^2 r^{2(k-1)}) \leq  r^2(\sum_{k=1}^{n} k^2\vert a_k \vert^2 r^{2(k-1)})=r^2(\int \vert p'(z) \vert^2 d{\bf m}_{a;r}) ,
$$

\noindent as we required. 
\end{proof} 

\noindent  In the case of the Lebesgue measure ${\bf m}$ we obtain the inequality with constant $C=1$, that is,
$$
\int \vert p(z)\vert^{2} d{\bf m} \leq \int \vert p'(z) \vert^{2} d{\bf m}, \qquad p(z) \in \mathbb{P}_0[z].
$$

\noindent Moreover, we prove  that if a measure on the unit circle $\nu$ verifies a Wirtinger polynomial inequality with constant $1$, i.e. 
$$
\int \vert p(z) \vert^{2} d\nu \leq 
\int \vert p(z) \vert^{2} d\nu, \qquad p(z) \in \mathbb{P}_0[z],
$$
\noindent then $\nu=c{\bf m}$ for some $C>0$. To do it, we use a matrix approach.

\begin{prop} Let $\mathbf{M}$ be an HPD matrix.  Let  $\mathbf{M}^{(1,1)}$ be the matrix obtained by  removing the first row and column in $\mathbf{M}$, i.e. $\mathbf{M}^{(1,1)}=(c_{i+1,j+1})_{i,j=0}^{\infty}$ Then, the following are equivalent: 
 
\begin{enumerate} 
\item 
$
\Vert p(z) \Vert^2_{\mathbf{M}} \leq  \Vert p'(z) \Vert^2_{\mathbf{M}}, \qquad p(z)\in \mathbb{P}_{0}[z].
$
\smallskip
\item 
$$
\mathbf{M}^{(1,1)} \leq \begin{pmatrix}  1 & 0& 0 & 0 & \dots \\  0 & 2& 0 & 0 & \dots \\  0 & 0 & 3 & 0 & \dots \\  \vdots  & \vdots & \vdots & \vdots & \ddots \end{pmatrix} \mathbf{M}^{(1,1)}
\begin{pmatrix}  1 & 0& 0 & 0 & \dots \\  0 & 2& 0 & 0 & \dots \\  0 & 0 & 3 & 0 & \dots \\  \vdots  & \vdots & \vdots & \vdots & \ddots \end{pmatrix}.
$$

\end{enumerate} 
\end{prop}

\begin{corollary} Assume that $\mathbf{M}$ is a diagonal matrix with positive entries. Then
$$
\mathbf{M}^{(1,1)} \leq \begin{pmatrix}  1 & 0& 0 & 0 & \dots \\  0 & 2& 0 & 0 & \dots \\  0 & 0 & 3 & 0 & \dots \\  \vdots  & \vdots & \vdots & \vdots & \ddots \end{pmatrix} \mathbf{M}^{(1,1)}
\begin{pmatrix}  1 & 0& 0 & 0 & \dots \\  0 & 2& 0 & 0 & \dots \\  0 & 0 & 3 & 0 & \dots \\  \vdots  & \vdots & \vdots & \vdots & \ddots \end{pmatrix}.
$$
\end{corollary} 

\begin{proof} Since $\mathbf{M}^{(1,1)}$ is a diagonal matrix,  it commutes with diagonal matrices. Thus,   the matrix inequality 
$$
 \begin{pmatrix}  c_{11} & 0& 0 & 0 & \dots \\  0 & c_{22}& 0 & 0 & \dots \\  0 & 0 & c_{33} & 0 & \dots \\  \vdots  & \vdots & \vdots & \vdots & \ddots \end{pmatrix} \leq  \begin{pmatrix}  c_{11} & 0& 0 & 0 & \dots \\  0 & 4c_{22}& 0 & 0 & \dots \\  0 & 0 & 9c_{33} & 0 & \dots \\  \vdots  & \vdots & \vdots & \vdots & \ddots \end{pmatrix},
$$

\noindent is equivalent to 
$$
 \begin{pmatrix}  0 & 0& 0 & 0 & \dots \\  0 & (4-1)c_{22}& 0 & 0 & \dots \\  0 & 0 & (9-1)c_{33} & 0 & \dots \\  \vdots  & \vdots & \vdots & \vdots & \ddots \end{pmatrix}
\geq 0,
$$

\noindent which is obviously true since  $c_{kk} \geq 0$. 
\end{proof}

\begin{prop} Let $\mathbf{T}$ be an HPD   Toeplitz matrix. Then the following are equivalent:
\medskip 

\begin{enumerate} 
\item $\mathbf{T}$ verifies a polinomial Wirtinger-type inequality with constant $C=1$, i.e. 
$$
\Vert p(z) \Vert^{2}_{\mathbf{T}} \leq 
\Vert p'(z) \Vert^{2}_{\mathbf{T}}, \qquad p(z) \in \mathbb{P}_0[z]. $$

\medskip
\item $\mathbf{T}=c{\bf I}$ for some $c>0$, where $\mathbf{I}$ is the identity matrix. \end{enumerate}  
\end{prop}

\begin{proof}  In order to prove $(1)$ implies $(2)$, assume   
$$
\mathbf{M}^{(1,1)} \leq \begin{pmatrix}  1 & 0& 0 & 0 & \dots \\  0 & 2& 0 & 0 & \dots \\  0 & 0 & 3 & 0 & \dots \\  \vdots  & \vdots & \vdots & \vdots & \ddots \end{pmatrix} \mathbf{M}^{(1,1)}
\begin{pmatrix}  1 & 0& 0 & 0 & \dots \\  0 & 2& 0 & 0 & \dots \\  0 & 0 & 3 & 0 & \dots \\  \vdots  & \vdots & \vdots & \vdots & \ddots \end{pmatrix}.
$$

\noindent In particular, for every $k >1$ it follows 

$$
\begin{pmatrix} c_{11} & c_{1k} \\ c_{k1} & c_{kk} \end{pmatrix} \leq \begin{pmatrix} c_{11} & kc_{1k} \\ kc_{k1} & k^2c_{kk} \end{pmatrix} \Leftrightarrow
\begin{pmatrix} 0 & (k-1)c_{1k} \\ (k-1)c_{1k} & (k^2-1)c_{kk} \end{pmatrix}
\geq 0.$$

\noindent Therefore, for $u,v\in \RR$ 

$$
\begin{pmatrix} u & v \end{pmatrix} 
\begin{pmatrix} 0 & (k-1)c_{1k} \\ (k-1)c_{k1} & (k^2-1)c_{kk} \end{pmatrix} \begin{pmatrix} u \\ v \end{pmatrix} = 2 \vert c_{1k}\vert^{2}(k-1)^2uv+ c_{kk} \vert u \vert^{2} \geq 0 ,
$$

\noindent and this is true if and only if $c_{1k}=0$. Since $\mathbf{T}$  is an HPD Toeplitz matrix then $c_{1k}=0$ implies that $c_{k1}=0$ 
and consequently if $i>  j$ $c_{ij}=c_{1,i-j+1}=0$. Then the matrix $\mathbf{T}=c_{00}\mathbf{I}$ as we required. 

\noindent  The other implication is a consequence of Lemma $3$ by taking in account that $\mathbf{I}$ is the moment matrix associated to  the Lebesgue measure ${\bf m}$. 
\end{proof} 
\begin{corollary} Let $\nu$ be a measure on $\mathbb{T}$. Then the following are equivalent: 
\begin{enumerate} 
\item $\nu$ verifies a polynomial Wirtinger-type inequality with constant $C=1$, i.e., 
$$
\int \vert p(z) \vert^{2} d\nu \leq \int \vert p'(z) \vert^{2} d\nu , \qquad p(z) \in \mathbb{P}_0[z]
$$
\item $\nu=c{\bf m}$ for some $c>0$. 

\end{enumerate}  
\end{corollary}  
\noindent Using the results in \cite{EGT1} concerning the behavior of the smallest and largest eigenvalues of the finite sections of an HPD matrix, we have:  

\begin{prop} Let $\mathbf{M}$ be an HPD matrix. Let $\lambda_n,\beta_n$ the smallest and largest eigenvalue of the finite section of order $n$ of $\mathbf{M}$. Assume that 
$$
0 < \lim_{n\to \infty} \lambda_n \leq \lim_{n\to \infty} \beta_n < \infty.$$
\noindent Then, $\mathbf{M}$ verifies a polynomial Wirtinger-type inequality, i.e. there exists $C>0$ such that 
$$
\Vert p(z) \Vert_{\mathbf{M}} \leq C\Vert p'(z) \Vert_{\mathbf{M}}, \qquad p(z) \in \mathbb{P}_0[z].
$$
\end{prop}
\begin{proof} Using the matrix approach as in  \cite{EGT1}, if $\lambda=\displaystyle{\lim_{n \to \infty}} \lambda_n$ and $\beta= \displaystyle{\lim_{n\to \infty} \beta_n}$ it follows that for every $p(z) \in \mathbb{P}[z]$,

$$
\lambda \Vert p(z) \Vert_{ \mathbf{I}}  \leq \Vert p(z) \Vert_{\mathbf{M}} \leq \beta \Vert p(z) \Vert_{\mathbf{I}}. 
$$
\noindent Then, 
$$
\Vert p(z) \Vert^{2}_{\mathbf{M}} \leq \beta^2 \Vert p(z) \Vert^{2}_{\mathbf{I}} \leq \Vert p'(z) \Vert^{2}_{\mathbf{I}} \leq \dfrac{\beta^2}{\alpha^{2}} \Vert p'(z) \Vert^{2}_{\mathbf{M}},
$$

\noindent for every $p(z)\in \mathbb{P}[z]$.
\end{proof}

\noindent In the cases of moment matrices associated with measures of the type $w(\theta)d{\bf m}$ it is well known (see, e.g. \cite{Grenader}) that $\displaystyle{\lim_{n\to \infty}\lambda_n}= {\it ess}\;  {\it inf} w(\theta)$ being ${\it ess}\;  {\it inf} w(\theta) $ the essential infimum of $w(\theta)$, that is, the largest $c>0$ such that $w(\theta) \geq c$ a.e. Analogously,  $\displaystyle{\lim_{n\to \infty}\beta_n}= {\it ess}\;  {\it sup} w(\theta)$ being ${\it ess}\;  {\it sup} w(\theta) $ the essential supremum of $w(\theta)$, that is, the smallest $C>0$ such that $w(\theta) \leq C$ a.e. Then we have the following. 

\begin{prop} Let $d\mu(z)=w(z)d{\bf m}$ with $w(z)\in L^{\infty}({\bf m})$ and ${\it ess inf} w(z)>0$. Then, there exists $C>0$ such that 
$$
\int \vert p(z) \vert^{2} d\mu \leq C\int \vert p'(z) \vert^{2} d\mu, \qquad p(z) \in \mathbb{P}_0[z].
$$
\end{prop}

\noindent Moreover, using Lemma 3: 
\begin{prop} Let $d\mu(z)=w(z)d{\bf m}_{a;R}$ with $w(z)\in L^{\infty}({\bf m}_{a;R})$ and ${\it ess inf} w(z)>0$. Then, there exists $C>0$ such that 
$$
\int \vert p(z) \vert^{2} d\mu \leq C\int \vert p'(z) \vert^{2} d\mu, \qquad p(z) \in \mathbb{P}_a[z],
$$
\noindent where $\mathbb{P}_a[z]=\{p(z)\in \mathbb{P}[z]: p(a)=0\}$.
\end{prop}

\section{Applications to boundedness of zeros of Sobolev orthogonal polynomials} 

\noindent  We here are interested in Sobolev inner products where the second measure is the Lebesgue measure ${\bf m}$, that is 
\begin{equation} \label{segunda-lebesgue}
\Vert p(z) \Vert_{\mathbf{S}}^{2} = \int \vert p(z) \vert^{2} d\mu +  \int \vert p'(z) \vert^{2} d{\bf m}.
\end{equation}

\noindent The key result is the following. 

\begin{lem} Let $\mu$ be a measure and consider the Sobolev inner product 
$
\Vert \cdot \Vert^{2}_{\mathbf{S}}$ in (\ref{segunda-lebesgue}).

\noindent  Assume that $0\in bpe(\mu)$, then there exists  $C>0$ such that 
$$
\int \vert p(z) \vert^{2} d{\bf m} \leq C \Vert p(z) \Vert_{\mathbf{S}}^{2}, \qquad p(z) \in \mathbb{P}[z].
$$
\end{lem}
\begin{proof} First, since $0\in bpe(\mu)$  there exists  $c>0$ such that for every $p(z)\in \mathbb{P}[z]$ it holds
$$
\vert p(0)\vert^{2} \leq c \int \vert p(z)\vert^{2} d\mu, \qquad p(z) \in \mathbb{P}[z].
$$
\noindent  Consider $p(z)\in \mathbb{P}[z]$, 
$$
\int \vert p(z) \vert^{2} d{\bf m} \leq \int (\vert p(z) -p(0) \vert+ \vert p(0) \vert)^{2}d{\bf m} \leq 2 \left( \int \vert p(z) -p(0) \vert^{2} d{\bf m} + \vert p(0) \vert^2\right) ,
$$

\noindent using  Lemma 3 it follows
$$ \int \vert p(z) \vert^{2} d{\bf m} \leq 
2(\int \vert p'(z) \vert^{2}d{\bf m}+ c\int \vert p(z) \vert^{2}d\mu) \leq C \Vert p(z) \Vert^{2}_{\mathbf{S}},
$$

\noindent  by  taking $C= 2\max \{c,1\}.$ 

\end{proof}

\begin{corollary}    Let $\mu$ be an infinitely supported  measure and consider
$
\Vert \cdot \Vert^{2}_{\mathbf{S}}$ as in (\ref{segunda-lebesgue}). Assume that $0\in bpe(\mu)$. Then $\mathcal{D}$ is bounded by $(\mathbb{P}[z], \Vert \cdot \Vert_{\mathbf{S}})$. Consequently, the set of zeros of the Sobolev orthogonal polynomials is bounded.
\end{corollary}

\noindent Using the results in \cite{EGT1} we obtain examples of vectorial measures which are not sequentially dominated, yet the matrix is sequentially dominated and the multiplication operator is bounded. 

\begin{prop} Let $\mu$ be any infinitely supported measure verifying:
\begin{enumerate} 
\item $0$ is a bpe of $\mu$.
\item $\displaystyle{\lim_{n\to \infty} \lambda_n=0}$, where $\lambda_n$ is the smallest eigenvalue of the $n\times n$ section of the moment matrix $\mathbf{M}(\mu)$
\end{enumerate}
\noindent Then
\begin{enumerate}
\item $\mathcal{D}$  is bounded with respect to the Sobolev norm associated with $(\mu,{\bf m})$.
\item $(\mu,{\bf m})$ is not sequentially dominated by the matrix, and consequently is not sequentially dominated.
\end{enumerate}
\end{prop}

\begin{ejem}
 Consider the family of vectorial measures $({\bf m}_r,{\bf m})$ with $0<r<1$. Then $\mathcal{D}$ is bounded with respect to the vectorial measure. However, $({\bf m}_r,{\bf m})$ is not sequentially dominated by the matrix. 
 \end{ejem} 

\bigskip 

\noindent More general results are obtained when we consider Lebesgue measures in circles ${\bf m}_{a;r}$ instead of the Lebesgue measure: 
\begin{prop} Let $\mu$ be an infinitely supported measure, and let $a_1,\dots,a_{n}\in bpe(\mu)$. Let $R_k>0$ for $k=1,\dots, n$ and $\mu_k=w_k(z)d\bf{m}_{a_k;R_k}$ with ${\it ess \; inf} w_k(z)>0$ and $w_k\in L^{\infty}({\bf m}_{a_k;R_k})$ for $k=1,\dots,n$.
Consider the Sobolev norm, 

$$
\Vert p(z) \Vert^{2}_{\mathbf{S}}= \int \vert p(z) \vert^{2} d\mu +  \sum_{k=1}^{n} \int \vert p'(z) \vert^{2} w_k(z)d{\bf  m}_{a_k;R_k}.
$$

\noindent Then $\mathcal{D}$ is bounded on $(\mathbb{P}[z], \Vert \cdot \Vert_{\mathbf{S}})$. Consequently, the set of zeros of the Sobolev orthogonal polynomials is bounded. 
\end{prop}

\begin{proof} We prove   condition (\ref{desigualdad}) for $(\mu_0,\mu)$ with   $\mu=\sum_{k=1}^{n} w_k(z)d{\bf  m}_{a_k;R_k}$. Thus, we have to find a constant $C>0$ such that 
$$
\int \vert p(z)\vert^2 d\mu \leq C(\int \vert p(z) \vert^2 d\mu_0+  \int \vert p'(z) \vert^{2} d\mu), \qquad p(z)\in \mathbb{P}[z].
$$

\noindent Let $k$ be fixed, since $p(z)=p(z)-p(a_k)+p(a_k)$, 
$$
\int \vert p(z) \vert^2 d\mu_1=\int \vert p(z) -p(a_k) + p(a_k) \vert^2 d\mu_1 \leq \int 2(\vert p(z) -p(a_k)\vert^2+\vert   p(a_k) \vert^2)\; d\mu_1.
$$

\noindent Thus, by taking $C_1=2\mu_1({\it supp}(\mu_1))$,

$$
\int \vert p(z) \vert^2 d\mu_1 \leq \sum_{k=1}^{n} \int 2\vert p(z) -p(a_k)\vert^2\; d\mu_1 + C_1\vert p(a_k) \vert^2.
$$ 

\noindent Next,  since $a_k$ is a bounded point evaluation of $\mu_0$ for $k=1,\dots,n$ there exists  $D>0$ such that 
$$
\sum_{k=1}^{n} \vert p(a_k) \vert^2 \leq D\int \vert p(z) \vert^2 \; d\mu_0 \leq D\Vert p(z) \Vert^2_{\mathbf{S}}.
$$

\noindent On the other hand, by 
the Wirtinger inequalities established in Proposition 11   for $\mu_k=w_k(z){\bf m}_{a_k;r_k}$ with  $k=1,\dots,n$, there exists $C>0$ such that 
$$
\sum_{k=1}^{n} \int \vert p(z) -p(a_k)\vert^2 w_k(z)d {\bf m}_{a_k;r_k} \leq  C \sum_{k=1}^{n} \int  \vert p'(z) \vert^2 d {\bf m}_{a_k;r_k} \leq C\Vert p(z) \Vert^{2
}_{\mathbf{S}}.
$$

\noindent Consequently, by taking $C'=2C_{1}D+2C$ we have that 

$$
\int \vert p(z) \vert^2 d\mu_1 \leq C(\int \vert p(z) \vert^2 d\mu_0+ \int \vert p'(z) \vert^2 d\mu_1 ), 
$$

\noindent as we required. 
\end{proof}



\noindent In the particular case where the first measure is a Lebesgue measure, we obtain the following example.
\begin{ejem} Let $\mu_0={\bf m}_{b;r}$ 
 and $a_1,\dots,a_n\in \mathbb{D}(b;r)$ for $k=1,\dots, n$.  consider $R_1,\dots, R_n>0$ and the Sobolev inner norm 
 $$
 \Vert p(z) \Vert^{2}_{\mathbf{S}} = \int \vert p(z) \vert^{2} d{\bf m}_{b;r}+ \sum_{k=1}^{n} \int \vert p'(z) \vert^{2} d{\bf m}_{a_k;R_k},
 $$
 \noindent $\mathcal{D}$ is bounded on $(\mathbb{P}[z], \Vert \cdot \Vert_{\mathbf{S}})$. Consequently, the 
  set of zeros of Sobolev orthogonal polynomials is bounded. 
\end{ejem}

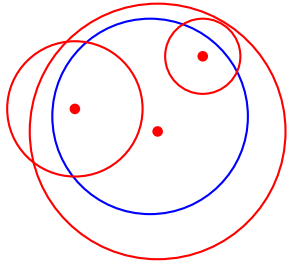
\begin{center}
 
\begin{tikzpicture}
    \draw[blue, thick] (-0.5,-0.3) circle (1.3cm);

    \coordinate (C1) at (0.2, 0.5); 
    \coordinate (C2) at (-1.5, -0.2); 
    \coordinate (C3) at (-0.4, -0.5); 
    
    \draw[red, thick] (C1) circle (0.5cm); 
    \draw[red, thick] (C2) circle (0.9cm); 
    \draw[red, thick] (C3) circle (1.7cm); 

    \fill[red] (C1) circle (2pt); 
    \fill[red] (C2) circle (2pt); 
    \fill[red] (C3) circle (2pt); 

\end{tikzpicture}

\end{center}
\begin{remark} It is interesting to compare the last example with example 5.  In both examples the first measure is the Lebesgue measure ${\bf m}_{b;r}$. While in the first example the circles in the second measure  have to be contained in $\mathbb{D}(b;r)$, however in the Example 6 the situation is more general: the centers must be  contained in such disk but the radius can be arbitrarily large. 
\end{remark}

\noindent We finish with a very interesting example that shows the existence of vectorial measures such that the Sobolev norms are equivalent norms on the vector space $\mathbb{P}[z]$ and nevertheless the norms associated to the involved measures are not equivalent component by component. In order to provide this example, we give the precise definitions of comparable norms.

\begin{definition} Let $\mu=(\mu_0,\mu_1)$  and $\mu'=(\mu'_0,\mu'_1)$ be  vectorial measure and let  $\Vert \cdot \Vert_{\mathbf{S}}, \Vert \cdot \Vert_{\mathbf{S'}}$ the associated Sobolev norms. 
\begin{enumerate}
 \item  We say that  $(\mu_0,\mu_1)$ is comparable with $(\mu'_0,\mu'_1)$, denoted by $(\mu_0,\mu_1) \;{}^{\vee}_{\wedge} (\mu_0',\mu'_1)$, if the norms  $\Vert \cdot \Vert_{\mathbf{S}}$ and $\Vert \cdot \Vert_{\mathbf{S'}}$ are equivalent on $\mathbb{P}[z]$, i.e, 
 if there exist  $c,C>0$  
 $$
 c\Vert p(z) \Vert^{2}_{\mathbf{S}} \leq \Vert p(z) \Vert^{2}_{\mathbf{S'}} \leq C\Vert p(z) \Vert^{2}_{\mathbf{S}}, \qquad p(z) \in \mathbb{P}[z]. 
 $$

 \noindent More generally, for HPD matrices $\{\mathbf{M}_0,\mathbf{M}_1\} \;{}^{\vee}_{\wedge} \{\mathbf{M}'_{0},\mathbf{M}'_1\}$ if the associated Sobolev norms $\Vert \cdot \Vert_{\mathbf{M_{S}}}, \Vert \cdot \Vert_{\mathbf{M}_{S'}}$ are equivalent on $\mathbb{P}[z]$.
 \item  We say that $(\mu_0,\mu_1)$ is comparable with $(\mu'_0,\mu'_1)$ {\it component by component}  if $\Vert \cdot \Vert_{\mu_k}$ is equivalent to $\Vert \cdot \Vert_{\mu'_k}$ for $k=0,1$, i.e. there exist $c,C>0$ such that for every  $k=0,1$, 
$$
c \Vert p(z) \Vert^{2}_{\mu_k} \leq \Vert p(z) \Vert^{2}_{\mu'_k} \leq C\Vert p(z) \Vert^{2}_{\mu_k}, \qquad p(z) \in \mathbb{P}[z].
$$
\end{enumerate}
\end{definition}

\noindent In the following example, we show that both notions for vectorial measures, comparable and comparable component by component, do not coincide.

\begin{ejem} There  are two vectorial measures
$\mu=(\mu_0,\mu_1)$  and $\mu'=(\mu'_0,\mu'_1)$ such are comparable and yet they are not comparable component by component;  moreover, $\mu'_0=\mu_0+\mu_1$ and $\mu'_1=\mu_1$. 
In fact, consider $\mu_0={\bf m}_{0;\frac{1}{2}}={\bf m}_{\frac{1}{2}}$ and $\mu_1={\bf m}$ and $\Vert \cdot \Vert_{\mathbf{S}},\Vert \cdot \Vert_{\mathbf{S'}}$ the associated Sobolev norms.

\begin{enumerate}
\item  $({\bf m}_{\frac{1}{2}},{\bf m}), ({\bf m}_{\frac{1}{2}},{\bf m}_{\frac{1}{2}}+{\bf m})$ are comparable. 

\noindent First, 
$$
\Vert p(z) \Vert^{2}_{\mathbf{S}}
\leq  
\Vert p(z) \Vert^{2}_{{\bf m}_{\frac{1}{2}}+{\bf m}} + \Vert p'(z) \Vert^{2}_{{\bf m}}=  \Vert p(z) \Vert^{2}_{\mathbf{S'}}, \qquad p(z) \in \mathbb{P}[z].
$$

\noindent On the other hand, 
$$
\Vert p(z) \Vert^{2}_{{\bf m}_{\frac{1}{2}}+{\bf m}} + \Vert p'(z) \Vert^{2}_{{\bf m}} = \Vert p(z) \Vert^{2}_{{\bf m}_{\frac{1}{2}}} + 
\Vert p(z) \Vert^{2}_{{\bf m}}+
\Vert p'(z)\Vert_{{\bf m}}, 
$$
\noindent Now, By Lemma 4, since $0$ is a bpe of ${\bf m}_{\frac{1}{2}}$ it holds that there exists $C>0$ such that 
$$
\Vert p(z) \Vert^{2}_{{\bf m}} \leq C\Vert p(z) \Vert^2_{\mathbf{S}}
$$

\noindent Consequently, 
$$
\Vert p(z) \Vert^{2}_{\mathbf{S}}\leq 
\Vert p(z) \Vert^{2}_{\mathbf{S'}}\leq (C+1) \Vert p(z) \Vert^{2}_{\mathbf{S}}, \qquad p(z) \in \mathbb{P}[z]. 
$$

\noindent Then, $(\mu_0,\mu_1) \;{}^{\vee}_{\wedge} (\mu_0,\mu_0+\mu_1).$ 

\medskip
\item  $({\bf m}_{\frac{1}{2}},{\bf m}), ({\bf m}_{\frac{1}{2}},{\bf m}_{\frac{1}{2}}+{\bf m})$ are not comparable {\it component by component}. Assume the contrary, then there exists a constant 
 $C>0$ such that 
$$
\int \vert p(z) \vert^2 d({\bf m}_{\frac{1}{2}}+ {\bf m }) \leq C\int \vert p(z) \vert^2{\bf m}_{\frac{1}{2}}, \qquad p(z)\in \mathbb{P}[z].
$$
\noindent Without lost of generality we may assume that $C>1$, consequently  
$$
\int \vert p(z) \vert^{2} d{\bf m} \leq (C-1) \int \vert p(z) \vert^2 d{\bf m}_{\frac{1}{2}}, \qquad p(z)\in \mathbb{P}[z],
$$
\noindent which obviously is not true (consider for instance the polynomials $z^n, n\in \NN$).

\end{enumerate} 

\end{ejem}

\noindent We bring together the previous results in the following table, where the four notions introduced for vectorial measures that appear inside are {\it not} equivalent.

   \hspace{-0.5 cm} 
   \begin{center}
   \begin{tabular}{|c|} \hline
\mbox{} \\
$ \mathcal{D}$ is bounded on $(\mathbb{P}[z], \Vert \cdot \Vert_{{\bf S}})\Leftrightarrow \{\mathbf{M}(\mu_0),\mathbf{M}(\mu_1)\} \;{}^{\vee}_{\wedge}\{ \mathbf{M}(\mu_0+\mu_1),\mathbf{M}(\mu_1)\}$
\mbox{} \\
\mbox{} \\
\begin{tabular}{|c|} \hline
\mbox{} \\
$ \mu=(\mu_0,\mu_1)  \;{}^{\vee}_{\wedge} (\mu_0+\mu_1,\mu_1)$ component by component  
\\
\mbox{} \\
\begin{tabular}{|c|} \hline
\mbox{} \\
$\mu=(\mu_0,\mu_1),\;\text{ is Matrix Sequentially Dominated i.e. }$  \\$\exists C>0 \;s.t.\mathbf{M}(\mu_1) \leq C\mathbf{M}(\mu_0)$
\\
 \mbox{} \\
\begin{tabular}{|c|} \hline
\mbox{} \\ $\mu=(\mu_0,\mu_1),\;  \text{is Sequentially Dominated if} $ \\ $\bullet{\it supp}(\mu_1)\subseteq {\it supp}(\mu_0)$\\$\bullet\; d\mu_1=f_{1} d\mu_{0};f_{1}\in L^{\infty}(\mu_{0})$
 \\
 \mbox{} \\
 \mbox{} \\ \hline
 \end{tabular}
 \mbox{} \\
 \mbox{} \\\hline
 \end{tabular}
 \mbox{} \\
 \mbox{} \\ \hline
 \end{tabular}
  \mbox{} \\
  \mbox{} \\ \hline
 \end{tabular}

\normalsize  
   \end{center}

\today

\end{document}